\documentclass[a4paper, 11pt]{article}
\usepackage{amsmath}
\usepackage{amssymb}
\usepackage{amsthm}
\usepackage{color}
\usepackage[width=460pt, height=650pt]{geometry}
\usepackage{authblk}
\usepackage{soul}
\usepackage{dsfont}
\usepackage{hyperref}
\hypersetup{
	colorlinks=true,
	linkcolor=black,
	anchorcolor=black,
	citecolor=black,
	filecolor=black,
	menucolor=red,
	runcolor=black,
	urlcolor=black,
}

\providecommand{\keywords}[1]
{
	\small	
	\textbf{\textit{Keywords:}} #1
}

\providecommand{\subjclass}[1]
{
	\small	
	\textbf{\textit{2020 Mathematics Subject Classification:}} #1
}

\DeclareMathOperator{\R}{\mathbb{R}}

\renewcommand{\P}{\mathbb{P}}

\newcommand{\N}{\mathbb{N}}

\newcommand{\uteta}{{\underline{\Theta}}}

\renewcommand{\phi}{\varphi}

\numberwithin{equation}{section}

\theoremstyle{plain}
\newtheorem{theorem}{Theorem}[section]

\newtheorem{corollary}[theorem]{Corollary}
\newtheorem{proposition}[theorem]{Proposition}
\newtheorem{lemma}[theorem]{Lemma}

\theoremstyle{remark}

\theoremstyle{definition}

\begin{document}

\title{Smoothness of random self-similar measures on the line and the existence of interior points}

\author[1]{Bal\'azs B\'ar\'any\thanks{B. B\'ar\'any was supported by the grants National Research, Development and Innovation Fund K142169 and KKP144059 “Fractal geometry and applications”.}}
\author[2,3]{Micha\l\ Rams\thanks{M. Rams was partially supported by National Science Centre
grant 2019/33/B/ST1/00275 (Poland) and by the Erd\H{o}s Center.\\ \indent The project was part of the semester program "Fractals and Hyperbolic Dynamical Systems" organized and funded by the Erd\H{o}s Center.}}
\affil[1]{Department of Stochastics, Institute of Mathematics, Budapest University of Technology and Economics, Műegyetem rpk. 1-3., Budapest, Hungary, H-1111}
\affil[ ]{email: barany.balazs@ttk.bme.hu}
\affil[2]{Institute of Mathematics, Polish Academy of Sciences, ul. \'Sniadeckich 8, 00-656 Warszawa, Poland}
\affil[3]{HUN-REN Alfr\'ed R\'enyi Institute of Mathematics, Re\'altanoda u. 13–15, H-1053 Budapest, Hungary}
\affil[ ]{email: rams@impan.pl}

\maketitle

\begin{abstract}
In this paper, we study the smoothness of the density function of absolutely continuous measures supported on random self-similar sets on the line. We show that the natural projection of a measure with symbolic local dimension greater than 1 at every point is absolutely continuous with H\"older continuous density almost surely. In particular, if the similarity dimension is greater than $1$ then the random self-similar set on the line contains an interior point almost surely.
\end{abstract}



\section{Introduction}

Let $\Phi=\{f_i\}_{i=1}^N$ be an iterated function system on the line, given by the maps $f_i(x) = \lambda_i x + t_i$. We denote by $\Lambda$ the attractor of the IFS, that is the unique nonempty compact set satisfying the equation
\[
\Lambda = \bigcup_{i=1}^N f_i(\Lambda),
\]
see Hutchinson~\cite{Hutchinson}.
The similarity dimension of the IFS is given by the solution of the equation
\begin{equation}\label{eq:simdim}
\sum_{i=1}^N \lambda_i^s=1.
\end{equation}
There is a natural way of presenting $\Lambda$ using the symbolic space $\Omega = \{1,\ldots,N\}^\N$. Namely, $\Lambda = \Pi(\Omega)$, where
\[
\Pi(\omega) =\lim_{n\to\infty}f_{\omega_1}\circ\cdots\circ f_{\omega_n}(0)=\sum_{j=0}^\infty \prod_{i=0}^{j-1} \lambda_{\omega_i} \cdot t_{\omega_j}.
\]
We will also denote
\[
\lambda_{\omega^j} = \prod_{i=0}^{j-1} \lambda_{\omega_i}
\]
for all words $\omega^j=(\omega_0, \ldots, \omega_{j-1})$ of length $j$. Using the contraction ratios of the IFS, one can define a natural metric on $\Omega$, namely,
$$
d(\omega,\tau)=\lambda_{\omega^{|\omega\wedge\tau|}},
$$
where $|\omega\wedge\tau|=\min\{n\geq0:\omega_{n+1}\neq\tau_{n+1}\}$. Let $\sigma$ denote the usual left-shift operator on $\Omega$. For finite words $\omega\in\Omega^*=\bigcup_{n=0}^\infty\{1,\ldots,N\}^n$, denote $|\omega|$ the length of $\omega$, furthermore, denote $[\omega]$ the corresponding cylinder set, that is,
$$
[\omega]=\{\tau\in\Omega:\omega_n=\tau_n\text{ for every }n\leq|\omega|\}.
$$

For the probability vector $(p_1,\ldots,p_N)$ and the corresponding Bernoulli measure $\mu_0$ on $\Omega$, let $\nu=\Pi_* \mu_0$ be the unique compactly supported probability measure such that
\[
\nu=\sum_{i=1}^Np_i(f_i)_*\nu
\]
called self-similar measure. With the choice of the probabilities $(\lambda_1^s,\ldots,\lambda_N^s)$, we call the measure $\mu$ the {\it natural measure}. Similarly, for any measure $\mu$ on $\Omega$ we can define its image $\nu=\Pi_* \mu$ on the real line.

Throughout the paper, we will denote by $\dim_H$ the Hausdorff dimension of sets and measures, for precise definition and basic properties, see for example \cite{BSS}. Furthermore, let us denote by $\mathcal{L}$ the Lebesgue measure on the line.

In the recent years, considerable attention has been paid for the dimension theory and geometric properties of self-similar sets and measures. {The combination of the results of Hutchinson \cite{Hutchinson} and Cawley and Mauldin \cite{CawleyMauldin} gives us the following:} if the IFS $\Phi$ satisfies the open set condition then
$$
\dim_H(\Lambda)=s\text{ and }\dim_H(\nu)=\frac{-\sum_{i=1}^Np_i\log p_i}{-\sum_{i=1}^Np_i\log|\lambda_i|},
$$
where $s$ is the similarity dimension. Hochman \cite{Hochman} generalised this result significantly for the overlapping situation and showed that if the exponential separation condition holds then
\begin{equation}\label{eq:typdim}
\dim_H(\Lambda)=\min\{1,s\}\text{ and }\dim_H(\nu)=\min\left\{1,\frac{-\sum_{i=1}^Np_i\log p_i}{-\sum_{i=1}^Np_i\log|\lambda_i|}\right\}.
\end{equation}
In particular, the exponential separation condition holds up to an $N-1$-dimensional family of translation parameters $(t_1,\ldots,t_N)$ if $\max_{i\neq j}\{|\lambda_i|+|\lambda_j|\}<1$, see Fraser and Shmerkin \cite{FraserShmerkin}. Jordan and Rapaport \cite{JordanRapaport} generalized \eqref{eq:typdim} for left-shift invariant ergodic measures, and Shmerkin \cite{Shmerkinlq} studied the $L^q$-dimension of self-similar measures under exponential separation.

It is a natural question to ask whether $s>1$ implies $\mathcal{L}(\Lambda)>0$, or if $\frac{-\sum_{i=1}^Np_i\log p_i}{-\sum_{i=1}^Np_i\log|\lambda_i|}>1$ implies $\mu\ll\mathcal{L}$, at least for typical choice of the natural parameters? This was verified by Saglietti, Shmerkin and Solomyak \cite{SSS}, namely, for every translation vector $(t_1,\ldots,t_N)$ such that $t_i\neq t_j$ and for Lebesgue-almost every contraction $(\lambda_1,\ldots,\lambda_N)$ with $\frac{-\sum_{i=1}^Np_i\log p_i}{-\sum_{i=1}^Np_i\log|\lambda_i|}>1$, the measure $\mu$ is absolutely continuous with respect to the Lebesgue measure. In particular, if $s>1$ then $\mathcal{L}(\Lambda)>0$ for typical choice of translation vectors.

Another natural question to ask whether $\mathcal{L}(\Lambda)>0$ then does $\Lambda$ have a non-empty interior? The answer for this question is negative on $\R^d$ for $d\geq2$, see Cs\"ornyei, Jordan, Pollicott, Preiss and Solomyak \cite{CsJPPS}, however, it is still open in $\R$. It is an open question even for typical choice of the parameters. {In the special case of the $\{0,1,3\}$-problem, Keane, Smorodinsky and Solomyak \cite{KeaneSmorodinskySolomyak} studied the problem of the existence of interior point, proving that there indeed exists certain $s_0>1$ such that if $s>s_0$ then the attractor has an internal point.}

{In this paper, we intend to study similar questions for randomly perturbed self-similar systems. Assume our base map is an iterated function system as above, and} let $\Theta$ be a uniformly bounded real-valued random variable on the real line, which will be our perturbation of the system. We will assume that $\Theta$ has a compactly supported absolutely continuous distribution and that the Fourier transform of its distribution (denoted by $\widehat{\Theta}$ with a slight abuse of notation) satisfies
\begin{equation}\label{eq:assfourier}
|\widehat{\Theta}(x)| \leq \frac C {(1+|x|)^M}
\end{equation}
for some constant $C>0$ and $M>0$ sufficiently large, depending on the deterministic IFS.

For every $\omega\in\Omega^*=\bigcup_{n=0}^\infty\{1,\ldots,N\}^n$, let $\Theta_{\omega}$ be independent copies of $\Theta$. Write $\uteta$ for the whole collection $\{\Theta_\emptyset, \Theta_1, \ldots, \Theta_N, \Theta_{11},\ldots\}$, and we equip the set $\mathbb{T}$ of all such sequences with the usual product topology. We will write $\P$ for the probability distribution on $\R^{\Omega_*}$ given by independent probability distribution of $\Theta$ on each vortex.

Let us consider the random set $\Lambda_\uteta=\Pi_\uteta(\Omega)$, given by
\[
\Pi_\uteta(\omega) = \sum_{j=0}^\infty \lambda_{\omega^j} \cdot (t_{\omega_j}+\Theta_{\omega^j}).
\]
For a measure $\mu$ on $\Omega$, denote also
\[
\nu_\uteta = (\Pi_\uteta)_*(\mu).
\]
In heuristic terms, we apply the perturbation independently at every vortex of the graph $\Omega_*$ consisting of the finite sequences of symbols $\{1,\ldots,N\}$.

This class of randomly perturbed systems has been already studied in the recent years. Jordan, Pollicott and Simon \cite{JPS07} studied the almost sure value of the dimension and Lebesgue measure of $\Lambda_{\underline{\Theta}}$, and the almost sure value of the dimension and the absolute continuity of $\nu_{\uteta}$ for ergodic left-shift invariant measures $\mu$. Dekking, Simon, Sz\'ekely and Szekeres \cite{DSSS} recently showed that if $s>1$ then $\Lambda_{\uteta}$ contains and interior point almost surely. Even more recently, Gu and Miao \cite{GuMiao} studied the $L^q$ dimension of randomly perturbed systems.

There are several other types of natural random perturbations of self-similar sets, like perturbing the contraction ratios, see for example Koivusalo \cite{Koivusalo} and Peres, Simon and Solomyak \cite{PSS}, which we do not discuss in this paper.

Let us now turn to our main result. We will in general assume that there exist $K>0$ and $s'>1$ such that
\begin{equation}\label{eq:assume}
	\mu([\omega]) \leq K \lambda_{\omega}^{s'}\text{ for every finite word $\omega\in\Omega^*$},
\end{equation}
that is the local dimension of $\mu$ is strictly greater than 1 everywhere in $\Omega$. We note that necessarily $s'\leq s$, where $s$ is the similarity dimension.

\begin{theorem}\label{thm:main}
	Let $\Phi$ be a self-similar IFS on the line such that $s>1$ defined in \eqref{eq:simdim}. Let $\mu$ be a measure on $\Omega$ such that \eqref{eq:assume} holds with some $s'>1$. Moreover, let $\Theta$ be a random variable such that \eqref{eq:assfourier} holds with $M\geq s'$. Then $\nu_\uteta$ is absolutely continuous with respect to the Lebesgue measure on the line and has H\"older continuous density almost surely.
\end{theorem}

An easy corollary of Theorem~\ref{thm:main} is the following:

\begin{corollary}\label{cor}
	Let $\Phi$ be a self-similar IFS on the line such that $s>1$ and let $\Theta$ be a random variable such that \eqref{eq:assfourier} holds with $M\geq s$. Then $\Lambda_\uteta$ contains an open interval almost surely.
\end{corollary}

First, let us discuss the assumptions of Theorem~\ref{thm:main}. Assumption $s>1$ is a natural assumption, since $s$ is a natural upper bound on the dimension of $\Lambda_\uteta$ for {\bf every} realisation of $\uteta$, thus, if $s<1$ then $\Lambda_\uteta$ has zero Lebesgue measure for every $\uteta$. Also, condition \eqref{eq:assume} is natural and close to optimal since if there exist $\omega\in\Omega$ and $s'<1$ such that $\mu([\omega^n])>\lambda_{\omega^n}^{s'}$ for infinitely many $n$ then
$$
\limsup_{n\to\infty}\frac{\nu_{\uteta}(B(\Pi_\uteta(\omega),\lambda_{\omega^n}))}{2\lambda_{\omega^n}}\geq\limsup_{n\to\infty}\frac{\mu([\omega^n])}{2\lambda_{\omega^n}}\geq\limsup_{n\to\infty}2^{-1}\lambda_{\omega^n}^{s'-1}=\infty,
$$
which means that $\mu$ cannot have bounded density. Moreover, \eqref{eq:assume} is strongly related to the lower-$L^q$ dimension of $\mu$. We define the lower-$L^q$ dimension $D(\mu,q)$ of $\mu$ for $q>1$ as the unique solution $D$ of the equation
$$
\liminf_{n\to\infty}\frac{-1}{n}\log\sum_{\omega\in\Omega_n}\mu([\omega])^q\lambda_{\omega}^{D(1-q)}=0.
$$
It is easy to see that \eqref{eq:assume} holds if and only if $\liminf_{q\to\infty}D(\mu,q)>1$. In particular, \eqref{eq:assume} implies that the lower $L^q$-dimension of $\nu_\uteta$ is $1$ for every $q>1$ almost surely, see Gu and Miao \cite[Theorem~2.9]{GuMiao}, which is necessary for the existence of continuous (bounded) density function.

Our only technical-like assumption is \eqref{eq:assfourier} on the distribution of random variables. The random perturbations are usually considered to be relatively smooth absolutely continuous random variables, so one may expect some kind of Fourier-decay of their distribution. {Nevertheless, the most natural perturbation of all -- with $\Theta$ given as a uniformly distributed measure on an interval, this is the class of systems investigated in \cite{DSSS} -- does not satisfy this assumption.}

In a sense, our main result can be considered as an extension of the recent results of Dekking, Simon, Székely and Szekeres \cite{DSSS} and Gu and Miao \cite{GuMiao}. The existence of the interior point is only a simple corollary of our main result, which was the main theorem in \cite{DSSS}, and our proof uses completely different methods. We rely on the Fourier transform of the density function and Kolmogorov's continuity theorem. On the other hand, we complete further smoothness properties of the projected measure rather than only the $L^q$-dimension.

Our methods borrows several ideas from the proof of Erraoui and Hakiki \cite[Theorem~3.5]{ErraouiHakiki}. We treat the local density of the random fractal measure at a given point $x$ as a random variable forming a part of a stochastic process, with $x$ playing the role of time. To do this we must first know that the random measure is almost surely absolutely continuous (hence the local density indeed exists almost everywhere, almost surely, and the stochastic process is well defined). We then estimate the increments of this process using the inverse Fourier transform and apply the Kolmogorov's Continuity Theorem. Unfortunately, this step where we use the inverse Fourier transform is only doable in dimension 1, as we are relying on the Carleson's Theorem which is not true in higher dimensions.








\

\section{Main tools}




Now, let us introduce the notation of the Fourier transform. If $\nu$ is a Borel probability measure on $\R$ and $f\colon\R\to\R$ an $L^1$-function then the Fourier transform of $\nu$ and $f$ is
$$
\hat{\nu}(\xi)=\int e^{ix\xi}d\nu(x)\text{ and }\hat{f}(\xi)=\int e^{ix\xi}f(x)dx.
$$
If $\nu$ is absolutely continuous with density function $f$ then $\hat{\nu}=\hat{f}$. One of our main tools is Carleson's Theorem.

\begin{theorem}[Carleson]\label{thm:carleson}
	Let $f$ be an $L^2$ function on $\R$ with respect to the Lebesgue measure. Then
	$$
	f(x)=\lim_{n\to\infty}\frac{1}{2\pi}\int_{-n}^n e^{-ix\xi}\hat{f}(\xi)d\xi\text{ for Lebesgue-almost every }x.
	$$
\end{theorem}

For the proof, see the book of Garafakos \cite[Theorem~11.1.1]{Grafakos}.

Our next tool is Kolmogorov's Continuity Theorem. We state here the special case we require.

\begin{theorem}[Kolmogorov's Continuity Theorem]\label{thm:kolmogorov}
	Let $(\Omega,\mathcal{F},\mathbb{P})$ be a probability space, and let $X=\{X_t\}_{t\in T}$ be a stochastic process (i.e. $X_t\colon\Omega\to\R$ is measurable for every $t\in T$), where $T\subset\R$ is bounded, and suppose that there exists $p>0$, $C>0$ and $\alpha>1$ such that for all $a,b\in T$
	$$
	\mathbb{E}(|X_a-X_b|^p)\leq C|a-b|^\alpha.
	$$
	Then $X$ has a (H\"older) continuous modification. That is, there exists a stochastic process $Y=\{Y_t\}_{t\in T}$ such that for every $t\in T$, $\mathbb{P}(X_t=Y_t)=1$ and the event $$\{t\mapsto Y_t\text{ is H\"older-continuous with any exponent smaller than} \ (\alpha-1)/p\}$$ is measurable and has full measure.
\end{theorem}

The proof of this version of Kolmogorov's Continuity Theorem can be found in Khoshnevisan's book \cite[Theorem~2.3.1 on page 158, Theorem~2.5.1 on page 166]{Khos}. Note that the formulation of Theorem 2.3.1 in the book has a typo, in the assumption ii) the term $r^{-1}$ is missing from the integral.




Finally, we state a well-known theorem from the theory of random IFS.

\begin{proposition}\label{prop:L2}
Let $\mu$ be a measure satisfying \eqref{eq:assume}. Then almost surely $\nu_\uteta$ is absolutely continuous with respect to the Lebesgue measure with $L^2$ density.
\end{proposition}

{By $L^2$ density, we mean here the $L^2$ space with respect to the Lebesgue measure.} The proof of the proposition can be found in Jordan, Pollicott and Simon \cite[Proposition~4.4(b)]{JPS07}, although it is not stated there that $\nu_\uteta$ has $L^2$ density, however, it is clear from the proof that indeed this is the case.

Finally we state a measurability lemma, which we need for the proof of the main theorem.

\begin{lemma}\label{lem:jointmeasurability}
	Let $(X,\mu)$ and $(Y,\nu)$ be compact separable metric spaces equipped with Borel measures. Assume a function $f:X\times Y\to\R$ satisfies
	\begin{itemize}
		\item for $\nu$-almost every $y\in Y$ $f(\cdot,y)$ is continuous,
		\item there exist a dense sequence $(x_i)_i\in Z\subset X$ such that for every $x_i$ $f(x_i,\cdot)$ is measurable.
	\end{itemize}
	Then $f$ is measurable.
\end{lemma}
\begin{proof}
	By Luzin Theorem for every $x_i$ we can write $Y=\bigcup_{n\in N} H_n(x_i) \cup Y_1(x_i)$, where $f(x_i,\cdot)$ is continuous in every (measurable) $H_n(x_i)$ and $\nu(Y_1(x_i)=0$. By Egorov Theorem we can present $Y = \bigcup_{n\in \N} G_n \cup Y_0$, where $G_n$ is the measurable set of $y\in Y$ for which the function $f(\cdot,y)$ is continuous with some uniform modulus of continuity $\delta_n$, and $\nu(Y_0)=0$ (observe that by continuity we only need to check the modulus of continuity on $Z\times Y$, and there it is a measurable function of $y$, hence we can indeed use the Egorov Theorem).
	
	Let $I\subset \R$ be an open interval and consider the set $A=f^{-1}(I)$. For any point $(x_i,y)\in A\cap Z\times Y; y\in G_{n_1}\cap H_{n_2}(x_i)$ let $a(x_i,y,I)$ be the distance from $f(x_i,y)$ to the endpoints of $I$. For any $\alpha\in (0,1)$ we define
	\[
	B(x_i, y, I, \alpha, n_1, n_2) = B(x_i, a_1) \times (B(y, a_2)\cap G_{n_1} \cap H_{n_2}(x_i)),
	\]
	where $a_2=a_2(x_i, a, n_2, \alpha)$ is such that $|f(x_i,y)-f(x_i,z)|<\alpha a(x_i,y,I)$ for all $z\in B(y, a_2) \cap H_{n_2}(x_i)$ and $a_1=a_1(x_i, a, n_1, \alpha)$ is such that $|f(x_i,z)-f(w,z)|<(1-\alpha) a(x_i,y,I)$ for all $w\in B(x_i, a_1)$ and $z\in G_{n_1}$. We define
	\[
	A' = \bigcup_{n_1, n_2} \bigcup_{(x_i,y)\in A\cap Z\times Y; y\in G_{n_1}\cap H_{n_2}(x_i)} \bigcup_{\alpha\in (0,1)}  B(x_i, y, I, \alpha, n_1, n_2)
	\]
	and observe that $A'\subset A$.
	
	Note now that $A'$ is a countable union (over $n_1, n_2, x_i$) of intersections of an open set $\bigcup_\alpha B(x_i,a_1)\times B(y,a_2)$ with a measurable set $X\times (G_{n_1}\cap H_{n_2}(x_i))$. Thus, $A'$ is measurable. Observe also that taking $\alpha \searrow 0$ we see that the whole interval $(x_i-\delta_{n_1}(a), x_i+\delta_{n_1}(a)) \times \{y\}$ is contained in $A'$ -- what's important here is that the length of this interval does not depend on $n_2$.
	
	Let us now consider the set $A\setminus A'$. We will claim that
	\[
	A\setminus A' \subset X \times (Y_0 \cup \bigcup_{x_i\in Z} Y_1(x_i)).
	\]
	As this set has zero measure, the assertion would follow from this claim.
	
	Assume that the claim is not true, and let $(x,y)\in A\setminus A'$ be such that $y\in G_{n_1}$ and that for every $x_i\in Z$ $y\in H_{n_2(x_i)}(x_i)$. Let $a'$ be the distance from $f(x,y)$ to the endpoints of $I$. As $Z$ is dense in $X$, there exist a subsequence $x_{m_i}\to x$. As $y\in G_{n_1}$, $f(x_{m_i},y)\to f(x,y)$, in particular $a(x_{m_i},y)\to a'>0$. This means that from some moment on $a(x_{m_i},y)> a'/2$, hence the whole interval $(x_{m_i}-\delta_{n_1}(a'/2), x_{m_i}+\delta_{n_1}(a'/2)) \times \{y\}$ will be contained in $A'$. As this interval contains the point $(x,y)$ for $m_i$ large enough, we get a contradiction. This proves the claim, and hence the lemma follows.
\end{proof}

\section{Absolute continuity with H\"older density}

Let us denote the density function of $\nu_\uteta$ by $\vartheta_\uteta$. By \cite[Theorem~2.12]{Mattila_geometry}, for $\psi$-almost every $t$
\begin{equation}\label{eq:density}
\vartheta_\uteta(x)=\lim_{r\to0}\frac{\nu_\uteta(B(x,r))}{2r}\text{ for Lebesgue-almost every $x$.}
\end{equation}

\begin{lemma}
The function $\vartheta_\uteta(x)$ is a measurable function of $(x,\uteta)$. 
\end{lemma}
\begin{proof}
We can write
\[
\vartheta_\uteta(x) = \lim_{r\searrow 0} Z_r(x,\uteta),
\]
where
\[
Z_r(x,\uteta) = \frac 1 {4r^3} \int_{r(1-r)}^{r(1+r)} \nu_\uteta(B(x,\ell)) d\ell
\]
is a continuous function of $(x,\uteta)$.

Indeed, the continuity follows from the definition of $\Pi_\uteta$. This map is even H\"older continuous if we equip $\mathbb{T}$ with the standard exponentially decreasing metric
\[
\rho(\uteta,\uteta') = \sum_{j=0}^\infty \sum_{\omega^j \in \{1,\ldots,N\}^j} \gamma^j |\uteta(\omega^j)-\uteta'(\omega^j)|
\]
for some $\gamma<1/N$. We then have
\begin{multline*}
(1-r) \cdot \frac {\nu_\uteta(B(x,r(1-r)))} {2r(1-r)}  =\frac {\nu_\uteta(B(x,r(1-r)))}{2r}\\
  \leq Z_r(x,\uteta) \leq \frac {\nu_\uteta(B(x,r(1+r)))} {2r} = \frac {\nu_\uteta(B(x,r(1+r)))} {2r(1+r)} \cdot (1+r).
\end{multline*}
\end{proof}

Let us denote the Fourier transform $\hat{\mu}$ of a Borel measure $\mu$ by
$$
\hat{\mu}(\xi)=\int e^{ix\xi}d\mu(x).
$$
Since by Proposition~\ref{prop:L2}, $\nu_\uteta\ll\mathcal{L}$ with $L^2$-density for $\P$-almost every $\uteta$, we get that $\hat{\nu}_\uteta=\hat{\vartheta_\uteta}\in L^2$. Then by Carleson's theorem (Theorem~\ref{thm:carleson}), for $\P$-almost every $\uteta$
\begin{equation} \label{eqn:carleson}
\vartheta_\uteta(x) = \lim_{n\to\infty} \frac 1 {2\pi} \int_{-n}^n e^{-ix\xi} \hat{\nu}_\uteta(\xi)d\xi\text{ for Lebesgue-almost every }x.
\end{equation}

Denote by $T$ the set of $(x,\uteta); x\in I, \uteta\in\mathbb{T}$ for which \eqref{eqn:carleson} holds.

\begin{lemma}\label{lem:timespace}
The set $T$ is Borel and has full (Lebesgue times $\P$) measure. In particular, there exists a measurable set $\mathcal{T}\subset I$ with full Lebesgue measure such that for every $x\in\mathcal{T}$, $T_x$ is measurable and has full $P$-measure. 
\end{lemma}

\begin{proof}
It is enough to check the measurability of $T$ and that $T$ has full $\mathcal{L}|_I\times\P$-measure. Then by Fubini's theorem (see for example Bogachev's book \cite[Theorem~3.4.1]{Bogachev}), for $\mathcal{L}|_I$-almost every $x$ the set $T_x=\{\uteta\in \mathbb{T}:(x,\uteta)\in T\}$ is measurable and the map $x\mapsto\P(T_x)$ is also measurable. We can then define $\mathcal{T}=\{x\in I:T_x\text{ is measurable and }P(T_x)=1\}$.

Denote the integral on the right hand side of \eqref{eqn:carleson} by
$$
f(n,x,\uteta)=\frac 1 {2\pi} \int_{-n}^n e^{-ix\xi} \hat{\nu}_\uteta(\xi)d\xi.
$$
Observe for future reference that $f(n,x,\uteta)\in \R$ for all $n,x\in\mathbb{R}$ and $\uteta\in \mathbb{T}$. Since $(x,\uteta)\mapsto f(n,x,\uteta)$ is clearly continuous, the set $\left\{(x,\uteta)\in I\times \mathbb{T}:|f(n_1,x,\uteta)-f(n_2,x,\uteta)|<1/N\right\}$ is clearly a Borel set for every $n_1,n_2\in\mathbb{Q}$. Then let
\[
T_1 = \bigcap_{N=1}^\infty \bigcup_{m=1}^\infty\bigcap_{\substack{n_1,n_2\in\mathbb{Q}\\ n_1,n_2>m}}\left\{(x,t)\in I\times \mathbb{T}:|f(n_1,x,\uteta)-f(n_2,x,\uteta)|<1/N\right\}
\]
be the set for which the right-hand side of \eqref{eqn:carleson} converges pointwise, which is clearly Borel. Similarly, let us define the set for which the right-hand side of \eqref{eq:density} converges. That is,
\begin{equation}\label{eq:thetameas}
T_2=\bigcap_{N=1}^\infty \bigcup_{m=1}^\infty\bigcap_{\substack{n_1,n_2\in\mathbb{Q}\\ 0<n_1,n_2<1/m}}\left\{(x,\uteta)\in I\times \mathbb{T}:\left|Z_{n_1}(x,\uteta)-Z_{n_2}(x,\uteta)\right|<1/N\right\}
\end{equation}
which is again clearly Borel. Finally, let $T_3$ be the set when the difference converges to zero, i.e.
$$
T_3=\bigcap_{N=1}^\infty \bigcup_{m=1}^\infty\bigcap_{\substack{n_1,n_2\in\mathbb{Q}\\ 0<n_1<1/m, m<n_2}}\left\{(x,\uteta)\in I\times \mathbb{T}:\left|Z_{n_1}(x,\uteta)-f(n_2,x,\uteta)\right|<1/N\right\}.
$$
Since $T=T_1\cap T_2\cap T_3$ is exactly the set of points for which \eqref{eqn:carleson} holds, it is measurable and by combining Fubini's, Marstrand's and Carleson's Theorem, it has full $\mathcal{L}|_I\times\P$-measure.
\end{proof}

Our goal is to show that there is a H\"older-continuous variant of $\vartheta_\uteta(x)$ over $\mathcal{T}$.

\begin{proposition}\label{prop:main}
There exists a function $(\uteta,x)\mapsto g_\uteta(x)$ such that
\begin{itemize}
\item the map $\uteta\in\mathbb{T}\mapsto g_\uteta(x)$ is measurable for every $x\in\mathcal{T}$,
\item set of $\uteta$ such that the event $\{x\in \mathcal{T}\mapsto g_\uteta(x)\text{ is H\"older-continuous}\}$ is measurable and has full $\P$-measure,
\item for every $x\in\mathcal{T}$, $\vartheta_\uteta(x)=g_\uteta(x)$ for $\P$-almost every $\uteta$.
\end{itemize}
\end{proposition}

\begin{proof}
First, let us fix constants $\gamma>0$ and $\alpha>1$, to be defined later. We also fix an even number $p>0$.

We want to use the Kolmogorov's Continuity Theorem (Theorem~\ref{thm:kolmogorov}). What we need is to check that for $\alpha>1$ and some constant $C>0$ (to be precised later) we have for every $a,b\in\mathcal{T}$,
\begin{equation}\label{eq:forcarl}
Z(a,b) :=\int (\vartheta_\uteta(a)-\vartheta_\uteta(b))^p dP(\uteta) \leq C \cdot |a-b|^{\alpha}.
\end{equation}

Let us calculate it,  using the notation $f(n,a,\uteta)$ defined above.
\[
Z(a,b) = \frac 1 {(2\pi)^p}\int \lim_{n\to\infty} (f(n,a,\uteta)-f(n,b,\uteta))^p d\P(\uteta).
\]
Applying the Fatou's Lemma (which we can do because $f(n,a,\uteta) \in \R$), we get that
\[
\begin{split}
Z(a,b)&\leq \liminf_{n\to\infty} \frac 1 {(2\pi)^p} \int (f(n,a,\uteta)-f(n,b,\uteta))^p d\P(\uteta)\\
&= \frac 1 {(2\pi)^p} \liminf_{n\to\infty} \int \int_{-n}^n \ldots \int_{-n}^n \prod_{k=1}^p (e^{-ia\xi_k}-e^{-ib\xi_k}) \hat{\nu}_\uteta(\xi_k)   d\xi_1 \ldots d\xi_p d\P(\uteta).\\
\end{split}
\]
Since the integrals are of finite support and the function $\prod_k (e^{-ia\xi_k}-e^{-ib\xi_k}) \hat{\nu}_t(\xi_k)$ is uniformly bounded, we may apply Fubini's theorem and we get that
\[\begin{split}
Z(a,b) &\leq\frac 1 {(2\pi)^p} \liminf_{n\to\infty}  \left|\int_{-n}^n \ldots \int_{-n}^n \prod_{k=1}^p (e^{-ia\xi_k}-e^{-ib\xi_k}) \int \prod_{k=1}^p \hat{\nu}_\uteta(\xi_k) d\P(\uteta) d\xi_1 \ldots d\xi_p\right|\\
&\leq \frac 1 {(2\pi)^p} \liminf_{n\to\infty} \int_{-n}^n \ldots \int_{-n}^n \prod_{k=1}^p |e^{-ia\xi_k}-e^{-ib\xi_k}| \cdot \left|\int \prod_{k=1}^p \hat{\nu}_\uteta(\xi_k) d\P(\uteta) \right|d\xi_1 \ldots d\xi_p\\
&=\frac 1 {(2\pi)^p}  \int_{-\infty}^\infty \ldots \int_{-\infty}^\infty \prod_{k=1}^p |e^{-ia\xi_k}-e^{-ib\xi_k}| \cdot \left|\int \prod_{k=1}^p \hat{\nu}_\uteta(\xi_k) d\P(\uteta) \right|d\xi_1 \ldots d\xi_p.
\end{split}
\]

Using estimation $|e^{-i\eta}-1| \leq C_\gamma |\eta|^\gamma$, we get
\begin{equation}\label{eq:Z1}
Z(a,b) \leq C |a-b|^{p\gamma} \int_{-\infty}^\infty \dots \int_{-\infty}^\infty \prod_{k=1}^p|\xi_k|^\gamma \cdot W(\xi_1,\ldots,\xi_p) d\xi_1 \ldots d\xi_p,
\end{equation}
where
\[
W(\xi_1,\ldots,\xi_p) = \left|\int \prod_{k=1}^p \hat{\nu}_\uteta(\xi_k) d\P(\uteta)\right|.
\]

Next we will estimate $W$. Substituting the definition of the Fourier transform and the formula for $\Pi_\uteta$, we get

\[\begin{split}
W(\xi_1,\ldots,\xi_p)&=\left|\iiint \prod_{k=1}^p e^{i\xi_k\Pi_\uteta(\omega^{(k)})}d\mu(\omega^{(1)}) \ldots d\mu(\omega^{(p)}) d\P(\uteta)\right|\\
&=\left|\iiint \prod_{k=1}^p e^{i\xi_1 \sum_{n=0}^\infty \lambda_{\omega^{(k),n}} (t_{\omega^{(k),n}}+\Theta_{\omega^{(k),n}})} d\P(\uteta) d\mu(\omega^{(1)})\ldots d\mu(\omega^{(p)})\right|\\
&=\left|\iiint \prod_{k=1}^p \prod_{n=0}^\infty e^{i\xi_k \lambda_{\omega^{(k),n}} \Theta_{\omega^{(k),n}}} d\P(\uteta) d\mu(\omega^{(1)})\ldots d\mu(\omega^{(p)})\right|.
\end{split}\]

 As every $\Theta_{\omega^{(k),n}}$ has the same distribution $\Theta$, we can write
\[
W(\omega^{(1)},\ldots,\omega^{(p)})(\xi_1,\ldots,\xi_p) := \prod_{n=0}^\infty \prod_{\tau^n} \hat{\Theta}(\lambda_{\tau^n} \cdot \sum_{k\in N_{\omega^{(1)},\ldots,\omega^{(p)}}(\tau^n)} \xi_k)
\]
where $N_{\omega^{(1)},\ldots,\omega^{(p)}}(\tau^n)$ is the set of $k$'s for which $\omega^{(k),n}=\tau^n$. We have thus

\[
W(\xi_1,\ldots,\xi_p) = \int W(\omega^{(1)},\ldots,\omega^{(p)})(\xi_1,\ldots,\xi_p)d\mu^p(\omega^{(1)},\ldots,\omega^{(p)}).
\]

As we have $p$ sequences $\omega^{(k)}$, we also have, for every $n$, exactly $p$ words $\omega^{(k),n}$ (some of which can coincide). Thus, for any $n$, in all the sets $N_{\omega^{(1)},\ldots,\omega^{(p)}}(\tau^n); \tau^n\in \Omega_n$ every symbol $k_\ell \in \{1,\ldots,p\}$ appears exactly once. Moreover, $\mu^p$-almost surely there exists $N_0$ such that for $n\geq N_0$ all the words $\omega^{(k),n}$ are different from all words $\omega^{(k'),n}, k'\neq k$.

We define a sequence $N_1,\ldots, N_p$ inductively. We take $N_1=0$, and then for every $k=2,\ldots,p$ we define $N_k$ as the smallest natural number such that $\omega^{(k),N_k}$ is different from every $\omega^{(k'),N_k}, k'<k$. We can write
\[
|W(\xi_1,\ldots, \xi_p)| \leq \sum_{j_1,\ldots, j_p} \mu^p(N_k=j_k \forall k) \cdot \prod_{k=1}^p \left|\hat{\Theta}\left(\lambda_{\omega^{(k),j_k}}\cdot\sum_{\ell\in N_{\omega^{(1)},\ldots,\omega^{(p)}}(\omega^{(k),j_k})}\xi_\ell\right)\right|.
\]

Observe that the sets $N(\omega^{(k),N_k})$ are all different and that the corresponding vectors $\sum_{\ell\in N(\omega^{(k),N_k})} e_\ell$ form a coordinate system in $\R^p$ with Jacobian bounded from below by a constant depending only on $p$ -- both those properties follow from $1,\ldots, k-1 \notin N(\omega^{(k),N_k})$. Thus, for a given sequence $N_1,\ldots, N_p$ we have

\[
\begin{split}
&\iiint \prod_{k=1}^p |\xi_k|^\gamma |W(\omega^{(1)},\ldots,\omega^{(p)})(\xi_1,\ldots, \xi_p)| d\xi_1\ldots d\xi_p \\
&\qquad\leq C \cdot \iiint \prod_{k=1}^p |\xi_k|^\gamma \frac 1 {(1+\lambda_{\omega^{(k),N_k}} \cdot \sum_{\ell\in N_{\omega^{(1)},\ldots,\omega^{(p)}}(\omega^{(k),N_k})} \xi_\ell)^M} d\xi_1\ldots d\xi_p\\
&\qquad = C' \cdot \prod_{k=2}^p \lambda^{-1-\gamma}_{\omega^{(k),N_k}},
\end{split}
\]
where we use the assumption
\[
|\hat{\Theta}(x)| < \frac C {(1+|x|)^M}
\]
and we also assume $M=s'>1+\gamma$.

Coming back to $Z(a,b)$, we get
\begin{equation} \label{eqn:rozne}
Z(a,b) \leq  C |a-b|^{p\gamma} \cdot \sum_Q \mu^p(Q) \prod_{k=1}^p \lambda^{-1-\gamma}_{\omega^{(k),N_k}},
\end{equation}
where $Q$ is the collection $(N_1, \omega^{(1),N_1}, \ldots, N_p, \omega^{(p),N_p})$.

The last step is to estimate the sum on the right hand side of \eqref{eqn:rozne}. We denote
\[
W_\ell= \sum_Q \mu^p(Q) \prod_{k=1}^\ell \lambda^{-1-\gamma}_{\omega^{(k),N_k}}
\]
and estimate it inductively. For $\ell=1$ we have $N_1=0$ and $W_1=1$. For $\ell=2$ we have
\[
\P(\lambda_{\omega^{(2),N_2}}<a_2| \omega^{(1)}) < K a_2^{s'}.
\]
Indeed, the event $\lambda_{\omega^{(2),N_2}}<a$ means that the diameter of the minimal cylinder $[\omega^{(1)}\wedge \omega^{(2)}]$ containing both $\omega^{(1)}$ and $\omega^{(2)}$ is smaller than $a$. In other words, there is the first cylinder $C_k(\omega^{(1)})$ of diameter smaller than $a$, and $\omega^{(2)}$ must belong to it. By our assumption about $\mu$, the probability $\mu$ that $\omega^{(2)}$ belongs to this cylinder is smaller than $Ka^{s'}$. Note that this estimation does not depend on $\omega^{(1)}$.

For general $\ell$ there is not one such cylinder but $\ell-1$, one around each $\omega^{(i)}, i<\ell$. Thus,
\[
\P(\lambda_{\omega^{(\ell),N_\ell}}<a_{\ell}| \omega^{(1)},\ldots, \omega^{(\ell-1)}) < (\ell-1)Ka_{\ell}^{s'},
\]
again independently of $\omega^{(1)},\ldots, \omega^{(\ell-1)}$, and as a result
\begin{equation} \label{eqn:calcula}
\P(\prod_{k=1}^p \lambda_{\omega^{(k),N_k}}<a) < \sum_{a_2\cdot \ldots \cdot a_p <a} (p-1)! \cdot K^{\ell-1} \cdot (a_2\cdot \ldots \cdot a_{p})^{s'}.
\end{equation}
Possible values of $a_i$ differ from each other by a factor at least $\max_j \lambda_j<1$, thus the number of collections $(a_2,\ldots, a_\ell)$ such that
\[
a > a_2\cdot\ldots\cdot a_\ell \geq a \max_j \lambda_j
\]
is of order at most $(\log a/\max_j \log \lambda_j)^{p-1}$. Thus, if $1+\gamma < s'$ then
\[
\begin{split}
W_p &< \sum_{m=1}^{\infty} (\max_j \lambda_j)^{m(-1-\gamma)} \P((\max_j \lambda_j)^m<\prod_{k=1}^\ell \lambda_{\omega^{(k),N_k}}\\
 &<(\max_j \lambda_j)^{m-1}) < \sum_{m=1}^{\infty} (\max_j \lambda_j)^{m(s'-1-\gamma)} O(m^{p-1}) <\infty.
 \end{split}
\]

We can now do the final choice of constants. We have $s'>1$, we choose $\gamma<s'-1$, and then we choose some even $p$ large enough that $p\gamma>1$. We are done.
\end{proof}

\begin{proof}[Proof of Theorem~\ref{thm:main}]
	
	Let $(x,\uteta)\mapsto g_\uteta(x)$ be te function defined in Proposition~\ref{prop:main}. Then by Lemma~\ref{lem:jointmeasurability}, $(x,\uteta)\mapsto g_\uteta(x)$ is measurable with respect to $\mathcal{L}|_I\times\P$. Hence, by the third observation of Proposition~\ref{prop:main} and Fubini's theorem, for $\P$-almost every $\uteta$
	$$
	\vartheta_{\uteta}(x)=g_{\uteta}(x)\text{ for $\mathcal{L}$-almost every }x\in\mathcal{T}.
	$$
	Since $\mathcal{T}$ has full Lebesgue measure and in particular, is dense in $I$, by redefining $\vartheta_{\uteta}(x)$ on an at most zero measure set, it is easy to see that for $\P$-almost every $\uteta$
	$$
	d\nu_{\uteta}(x)=g_\uteta(x)d\mathcal{L}(x),
	$$
	which completes the proof.
\end{proof}

\begin{proof}[Proof of Corollary~\ref{cor}]
	By Theorem~\ref{thm:main} applied for the Bernoulli measure $\mu$ with probability vector $(\lambda_1^s,\ldots,\lambda_N^s)$, where $s$ is defined in \eqref{eq:simdim}, $\nu_\uteta=(\Pi_\uteta)_*\mu$ is absolutely continuous with H\"older-continuous density $\P$-almost surely. Since $\nu_\uteta$ is a probability measure, there has to be a point $x\in\Lambda_\uteta$ for which the density $\vartheta_\uteta(x)>0$. Using the continuity of $\vartheta_\uteta(x)$, the support of $\nu_{\uteta}$ (which is $\Lambda_\uteta$) must contain an interior point.
\end{proof}

\bibliographystyle{abbrv}
\bibliography{selfsimilar_bib}

\end{document}